\documentclass[a4paper,12pt]{article}
\usepackage{amsfonts}
\usepackage{pifont}
\usepackage{caption}
\usepackage{graphicx, subfig}
\usepackage{bm}
\usepackage{latexsym,amsmath,amssymb,cite,amsthm}
\usepackage{color,eucal,enumerate,mathrsfs}
\usepackage[normalem]{ulem}
\usepackage{amsmath}
\usepackage[pagewise]{lineno}
\usepackage[pagebackref=true, colorlinks, linkcolor=blue,  anchorcolor=blue,citecolor=blue]{hyperref}

\numberwithin{equation}{section}
\theoremstyle{plain}
\newtheorem{theorem}{Theorem}[section]

\newtheorem{lemma}[theorem]{Lemma}
\newtheorem{example}[theorem]{Example}
\newtheorem{proposition}[theorem]{Proposition}
\theoremstyle{definition}
\newtheorem{condition}[theorem]{Condition}
\newtheorem{definition}[theorem]{Definition}
\newtheorem{assumption}[theorem]{Assumption}
\theoremstyle{remark}
\newtheorem{remark}[theorem]{Remark}

\newcommand\bdf{\begin{definition}}
\newcommand\bpr{\begin{proposition}}
\newcommand\brk{\begin{remark}}
\newcommand\blm{\begin{lemma}}
\newcommand\bexe{\begin{exercise}}
\newcommand\bexa{\begin{example}}
\newcommand\beqn{\begin{eqnarray*}}
\newcommand\edf{\end{definition}}
\newcommand\epr{\end{proposition}}
\newcommand\erk{\end{remark}}
\newcommand\elm{\end{lemma}}
\newcommand\eexe{\end{exercise}}
\newcommand\eexa{\end{example}}
\newcommand\eeqn{\end{eqnarray*}}

\newcommand{\lmt}[2]{\mathop{\lim}_{{#1} \rightarrow {#2}} }

\newcommand{\lmts}[2]{\mathop{\overline{\lim}}_{{#1} \rightarrow {#2}} }
\newcommand{\lmti}[2]{\mathop{\underline{\lim}}_{{#1} \rightarrow {#2}} }

\newcommand{\mm}{\mathfrak m}
\newcommand{\ms}{(X,\d,\mm)}

\newcommand{\cdkn}{{\rm CD}(K, N)}
\newcommand{\mcpkn}{{\rm MCP}(K, N)}

\newcommand{\C}{\mathfrak{C}}

\newcommand{\E}{\mathcal{E}}

\newcommand{\N}{\mathbb{N}}

\newcommand{\R}{\mathbb{R}}

\newcommand{\diam}{\mathop{\rm diam}\nolimits} 

\renewcommand{\d}{{\mathrm d}}

\newcommand{\dt}{{\d t}}

\newcommand{\restr}[1]{\lower3pt\hbox{$|_{#1}$}}

\newcommand{\eps}{\varepsilon}
\newcommand{\nchi}{{\raise.3ex\hbox{$\chi$}}}

\title{{\bf Maz'ya\textendash Shaposhnikova meet Bishop\textendash Gromov}
}
\begin{document}

\author{Bang-Xian Han\thanks{School of   Mathematics, Shandong University, hanbx@sdu.edu.cn}
\and Andrea Pinamonti
\thanks{Dipartimento di Matematica, Universit\`a di Trento, andrea.pinamonti@unitn.it}
\and Zhefeng Xu
\thanks{School of   Mathematical Sciences, University of Science and Technology of China,  xzf1998@mail.ustc.edu.cn}
\and Kilian Zambanini
\thanks{Dipartimento di Matematica, Universit\`a di Trento, kilian.zambanini@unitn.it}}

\date{\today}
\maketitle

\begin{abstract}
We find a surprising link between Maz'ya\textendash Shaposhnikova's  well-known asymptotic  formula concerning fractional Sobolev seminorms and the generalized Bishop--Gromov inequality. In the setting of abstract metric measure spaces we prove the validity of a large family of asymptotic formulas concerning non-local energies. Important examples 
which are covered by our approach are for instance Carnot groups, Riemannian manifolds with Ricci curvature bounded from below and non-collapsed RCD spaces. We also extend the classical Maz'ya\textendash Shaposhnikova's formula on Euclidean spaces to a wider class of mollifiers.
\end{abstract}

\maketitle


\section{Introduction}\label{first section}
Let $N\in \mathbb{N}$, $p\geq 1$ and $0<s<1$. We recall that the fractional Gagliardo seminorm of a measurable function $f:\R^N\to\mathbb{R}$ is defined as
\[
\|f\|_{W^{s, p}(\mathbb{R}^N)}:=\left (\int_{\R^N} \int_{\R^N} \frac{|f(x)-f(y)|^p}{|x-y|^{N+sp}} \,\d\mathcal L^N (x) \d \mathcal L^N (y) \right)^{\frac 1p}.
\]
Here $W^{s, p}(\mathbb{R}^N)$ denotes the classical fractional Sobolev space, i.e. the set of all $L^p$ functions $f$ s.t. $\|f\|_{W^{s, p}(\mathbb{R}^N)}<\infty$.
In the celebrated paper \cite{BBM}, the authors studied the asymptotic behaviour of  $\|\cdot\|_{W^{s, p}(\R^N)}$ as $s\to 1^{-}$ for $p>1$. In particular, they proved that whenever $f\in W^{1,p}(\mathbb{R}^N)$ the following formula holds:
\begin{equation}\label{BBM:intro}
 \mathop{\lim}_{s \uparrow 1} ~(1-s)\|f\|^p_{W^{s, p}(\R^N)}= C  \|  \nabla f\|_{L^p(\R^N)}^p, \tag{BBM}
\end{equation}
where $C>0$ is an explicit constant depending only on $N$ and $p$. Afterwards, D\'avila \cite{MR1942130} generalized this result to functions of bounded variation (BV functions)  and their variation measures.
Later,  Maz'ya and Shaposhnikova \cite{MS} studied the asymptotic behaviour of the Gagliardo seminorm when $s\to 0^+$. They proved that for any $f\in \cup_{0<s<1} W^{s, p}(\R^N)$
 \begin{equation}\label{MS:intro}
 \mathop{\lim}_{s \downarrow 0} ~s\|f\|^p_{W^{s, p}(\R^N)}= {L }  \|  f\|_{L^p(\R^N)}^p, \tag{MS}
\end{equation}
where  $L$ is another positive constant depending only on $p\geq 1$ and  $N\in \N$.

Several different generalizations of these results in Euclidean spaces
have been considered and the literature became extremely vast; a detailed overview is beyond the scope of our contribution. Here, we restrict ourselves to the research that is most close in spirit to \cite{BBM} and \cite{MS}. Interesting contributions in this directions have been provided by Ponce \cite{Ponce}, Leoni--Spector \cite{MR2832587,MR3132740}, Brezis--Nguyen \cite{MR4011115,MR3749763, MR3556344,MR3485124}, 
Brezis--Van Schaftingen--Yung \cite{MR4482042,MR4275122},
Nguyen--Pinamonti--Vecchi--Squassina \cite{Ngu,MR3601583}, Buseghin--Garofalo--Tralli \cite{MR4685023, MR4453966}, Bru\'e--Calzi--Comi--Stefani \cite{MR4449863}, Ali--Lam--Pinamonti \cite{MR3886626,MR3986928},  Arroyo--Rabasa--Bonicatto \cite{MR4507709}, Brazke--Schikorra--Po Lam \cite{MR4525722} and references therein.
On the other hand, in the past three decades, there has been significant progress on the study of various aspects of first order analysis on metric measure spaces including the theory of first order Sobolev spaces, functions of bounded variation and their relation to variational problems and partial differential equations, see e.g. \cite{AT-T} and references therein. 
In \cite[Remark 6]{MR1942116}, Brezis raised a question about the relation between the BBM formula and Sobolev spaces on metric measure spaces. Di Marino--Squassina \cite{DMS19} proved new characterizations of Sobolev and BV spaces in PI spaces in the spirit of the BBM formula, see also \cite{LPZ1,LPZ2}. A similar result was proved previously in Ahlfors-regular spaces by Munnier \cite{MR3299669}. G\'orny \cite{MR4375837} studied the problem in certain PI spaces that ``locally look like"  Euclidean spaces. 
As far as we know the validity of \eqref{MS:intro} in metric measure spaces is still an open problem. The aim of the present paper is to fill this gap.  Moreover, we find an intriguing relation between the validity of \eqref{MS:intro} and suitable geometric properties of the underlying metric space. 

In order to describe this relation we recall that in comparison geometry, no matter in smooth or non-smooth setting,  \emph{Bishop--Gromov type inequalities} play a central role. For curvature parameter $K$ and dimension parameter $N$,  denote by $V^{K, N}(R)$ the volume of the geodesic ball with radius $R$ in the space form. We say that a metric measure space $\ms$ satisfies  the generalized Bishop--Gromov inequality if
\begin{equation}\label{bgi}\tag{BGI}
 \frac {\mm\big (B_R(x_0)\big)} {V^{K, N}(R)} \leq \frac {\mm\big (B_r(x_0)\big)} {V^{K, N}(r)}~~~~\text{for all}~~ R\geq r.
\end{equation}
In the present paper we find a surprising link between the Maz'ya\textendash Shaposhnikova's formula \eqref{MS:intro} and  the generalized Bishop--Gromov inequality \eqref{bgi}. Using this simple,  precise and geometric information we provide several new interesting examples  satisfying such an asymptotic  formula.  
In particular, applying our result to the metric measure spaces with synthetic curvature-dimension condition \`a la Lott\textendash Sturm\textendash Villani,  we  find a couple of new  sharp functional inequalities which, as far as we know, are new even in the smooth setting.\\

Another interesting outcome of the present paper is the extension of the formula \eqref{MS:intro}  with mollifier $\frac{s}{
\d(x, y)^{N+sp}}$, which appears in the Gagliardo
seminorm,  to the case of general radial mollifier satisfying suitable monotonicity assumptions (cf. Condition \ref{assumption3}).
In \cite{Davoli} the authors devised necessary and sufficient conditions on the mollifiers to get formula \eqref{BBM:intro}. The analogous problem for formula \eqref{MS:intro} is open and it will be the subject of future research.


\medskip

\noindent \textbf{Acknowledgement.} The first and  the third author are  supported by the Ministry of Science and Technology of China, through the Young Scientist Programs (No. 2021YFA1000900  and 2021YFA1002200), and NSFC grant (No.12201596). The second and the fourth author are supported by the University of Trento and the MUR-PRIN 2022 - Project: Regularity problems in Sub-Riemannian
structures, Project code: 2022F4F2LH.

\section{Main results}
\subsection{General theory}\label{sect: mainth}
In this paper,  a metric measure space  is a triple $\ms$, where $(X, \d)$ is a complete and separable metric space and  $\mm$ is a locally finite non-negative Borel measure on $X$ with full support.

First of all, we set some general assumptions on the mollifiers. 

\begin{assumption}\label{assumption1}
Let $\ms$ be a metric measure space and let $p\geq 1$. Suppose there exists a sequence of non-negative, symmetric, measurable  functions $(\rho_n)_{n \in \N}$  defined on $\{(x,y)\in X\times X: x\neq y\}$, called mollifiers,  such that 
 \begin{itemize}
\item [A)] there exists a constant $L\geq 0$ such that,  for any $x\in X$,
\begin{equation}\label{eq1:assumption1}
\lmt{R}{+\infty} \lmts{n}{\infty}\int_{B^c_R(x) }  \rho_n(x, y) \,\d \mm(y)=\lmt{R}{+\infty} \lmti{n}{\infty}\int_{B^c_R(x) }  \rho_n(x, y) \,\d \mm(y)=L,
 \end{equation}
 where $B^c_R(x)=\{y\in X: \d(x,  y)\geq R\}$;
 \item [B)]  for any $u \in L^p$ such that there exists $n_0\in\N$ with
\begin{equation*}
\E_{n_0}(u):= \iint_{\{(x, y):x, y\in X, x\neq y\}} {|u(x)-u(y)|^p} \rho_{n_0}(x, y)\,\d\mm(x) \d\mm(y)<+\infty,
\end{equation*}
we  have
 \begin{equation}\label{eq2:assumption1}
 \lmt{n}{\infty} \iint _{\big\{(x, y): 0<\d(x, y)<R\big\}} {|u(x)-u(y)|^p} \rho_n(x, y)\,\d\mm(x) \d\mm(y)=0,~~\forall R>0;
\end{equation}
\item[C)]for any $R>0$ sufficiently large, there exists a constant $C=C(R)$ such that, for any $x\in X$, $n\in \N$,
\begin{equation}\label{eq3:assumption1}
    \int_{B^c_R(x)}\rho_n(x,y)\,\d\mm(y)\leq C.
\end{equation}
\end{itemize}

\end{assumption}
 \medskip
The main result of this paper is the following Maz'ya\textendash Shaposhnikova type formula which is expressed in a very general form.

\begin{theorem}[Generalized  Maz'ya\textendash Shaposhnikova's formula]\label{th:MS2}
 Let $\ms$ be a metric measure space and let $(\rho_n)_{n\in \N}$ be mollifiers satisfying Assumption \ref{assumption1}.
Then, for any $u \in L^p$ such that
$\E_{n_0}(u)<+\infty$ for a certain $n_0\in \N$, it holds
\begin{equation}\label{eq1:MS}\tag{GMS}
\mathop{\lim}_{n\to \infty} \E_n(u)=2L\|u\|^p_{L^p}.
\end{equation}
\end{theorem}

\begin{proof}

 Fix $x_0 \in X$ and $R > 0$. We consider the following decomposition of  $X$:
 \begin{equation*}
\left\{
             \begin{array}{l}
             {A:=\big \{(x, y): 0<\d(x, y)< R \big \}} \\
             {B:= \big \{(x, y): \d(x, y)\geq R,  \,\d(y, x_0)>2\d(x, x_0)~\text{or}~\d(y, x_0)< \frac 12 \d(x, x_0) \big \}} \\
               {C:=   \big \{(x, y): \d(x, y)\geq R, \, \frac 12  \d(x, x_0) \leq \d(y, x_0)\leq 2\d(x, x_0)  \big \}}.
             \end{array}  
        \right.
\end{equation*}
so that
 \begin{eqnarray*}
\E_n(u) &=& \underbrace{  \int_{A} {|u(x)-u(y)|^p}{\rho_n(x, y)} \,\d \mm( x)\,\d \mm(y)}_{I:=I(R, n)}\\
&& + \underbrace{  \int_{{B} } {|u(x)-u(y)|^p}{\rho_n(x, y)} \,\d \mm(x)\,\d \mm(y)}_{II:=II(R, n)}\\
&& + \underbrace{ \int_{{C} } {|u(x)-u(y)|^p}{\rho_n(x, y)} \,\d \mm(x)\,\d \mm(y)}_{III:=III(R, n)}.
\end{eqnarray*}
By  Assumption \ref{assumption1}-B) we have
\begin{equation}\label{eq1.1:MS}
\lmt{R}{+\infty}\lmt{n}{+\infty} I(R, n) =0.
\end{equation}
For any  $x, y\in X$ satisfying $\d(y, x_0)>2\d(x, x_0) $,  by the triangle inequality we know 
\[
\d(x, y) \geq \d(y, x_0)-\d(x_0, x) >\d(y, x_0)-\frac 12 \d(y, x_0)=\frac 12 \d(y, x_0)
\]
and
\[
\d(x, y)  \leq \d(x_0, x) +\d(y, x_0) < \frac32 \d(y, x_0).
\]
Therefore
\[
\Big \{x:  \d(x, y)\geq R,~ \d(y, x_0)>2\d(x, x_0) \Big \} \subset  \Big  \{x:  \frac12 \d(y, x_0)\vee R  \leq  \d(x, y)<  \frac32 \d(y, x_0) \Big \}
\]
so that
\begin{eqnarray*}
II_a(R, n)&:=&\int_{X} |u(y)|^p \bigg( \int_{\big \{x: ~\d(x, y)\geq R,~ \d(y, x_0)>2\d(x, x_0) \big \}} {\rho_n(x, y)} \,\d \mm(x) \bigg) \,\d \mm(y)\\
&\leq &\int_{X} {|u(y)|^p} \bigg( \int_{\big  \{x:~  \frac12 \d(y, x_0)\vee R  \leq  \d(x, y)<  \frac32 \d(y, x_0)\big \}}  {\rho_n(x, y)} \,\d \mm(x) \bigg ) \,\d \mm(y)\\
&=&    \int_{X} {|u(y)|^p} \bigg(\int_{ \big  \{x: ~ \d(x, y)\geq \frac12 \d(y, x_0)\vee R \big \}}  {\rho_n(x, y)} \,\d \mm(x)\\
&&-\int_{ \big  \{x:  ~ \d(x, y)\geq \frac32 \d(y, x_0)\vee R \big \}}  {\rho_n(x, y)} \,\d \mm(x)   \bigg)  \,\d \mm(y).
\end{eqnarray*}
Note that, for $R$ sufficiently large, $\rho_n(\cdot,y)$ is integrable on the complementary of each ball centered at $y$ by Assumption \ref{assumption1}-C and the symmetry of $\rho_n$.\\
Moreover, for $R$ sufficiently large, 
\[\begin{split}&\,\int_{ \big  \{x: ~ \d(x, y)\geq \frac12 \d(y, x_0)\vee R \big \}}  {\rho_n(x, y)} \,\d \mm(x)
-\int_{ \big  \{x:  ~ \d(x, y)\geq\frac32 \d(y, x_0)\vee R \big \}}  {\rho_n(x, y)} \,\d \mm(x)\\ \leq & \,\int_{ \big  \{x: ~ \d(x, y)\geq \frac12 \d(y, x_0)\vee R \big \}}  {\rho_n(x, y)} \,\d \mm(x)
+\int_{ \big  \{x:  ~ \d(x, y)\geq \frac32 \d(y, x_0)\vee R \big \}}  {\rho_n(x, y)} \,\d \mm(x)\\
\leq &\, 2\int_{\big\{x: ~ \d(x,y)\geq R\big\}}\rho_n(x,y) \,\d\mm(x)\leq 2C.\end{split}\]
Hence, by Fatou's lemma,
\[\begin{split}
    \lmts{n}{\infty} II_a(R, n)\leq \int_X|u(y)|^p\lmts{n}{\infty}\bigg(&\int_{ \big  \{\d(x, y)\geq \frac12 \d(y, x_0)\vee R \big \}}  {\rho_n(x, y)} \,\d \mm(x)\\ &
-\int_{ \big  \{\d(x, y)\geq\frac32 \d(y, x_0)\vee R \big \}}  {\rho_n(x, y)} \,\d \mm(x)   \bigg)  \,\d \mm(y).
\end{split}\]
By a similar argument, since the map $R\mapsto C(R)$ can be chosen to be non-increasing for large $R$, we get
\begin{small}\[\begin{split}\lmt{R}{+\infty}\lmts{n}{\infty}II_a(R, n)&\leq \int_X|u(y)|^p\lmt{R}{+\infty}\lmts{n}{\infty}\bigg(\int_{ \big  \{\d(x, y)\geq\frac12 \d(y, x_0)\vee R \big \}}  {\rho_n(x, y)} \,\d \mm(x)\\ &\qquad
-\int_{ \big  \{\d(x, y)\geq\frac32 \d(y, x_0)\vee R \big \}}  {\rho_n(x, y)} \,\d \mm(x)   \bigg)   \,\d \mm(y)\\
&\leq \int_X|u(y)|^p\bigg(\lmt{R}{+\infty}\lmts{n}{\infty}\int_{ \big  \{\d(x, y)\geq\frac12 \d(y, x_0)\vee R \big \}}  {\rho_n(x, y)} \,\d \mm(x)\\ &\qquad
-\lmt{R}{+\infty}\lmti{n}{\infty}\int_{ \big  \{\d(x, y)\geq\frac32 \d(y, x_0)\vee R \big \}}  {\rho_n(x, y)} \,\d \mm(x)   \bigg )  \,\d \mm(y).\end{split}\]\end{small}\\
By Assumption \ref{assumption1}-A) we infer
\begin{equation}\label{eq1.11:MS}
\lmt{R}{+\infty} \lmts{n}{\infty} II_a(R, n)=0.
\end{equation}
Similarly, for $x\neq x_0$, it is immediate to see that
\[
B^c_{4\d(x, x_0)}(x) \subset \Big \{y\in X: \d(y, x_0)>2\d(x, x_0) \Big \}\subset B^c_{\d(x, x_0)}(x),
\]
so that
\begin{eqnarray*}
II_b(R, n)&:=& \int_{X} |u(x)|^p \bigg( \int_{ \big \{y:~\d(x,y)\geq R, ~ \d(y, x_0)>2\d(x, x_0) \big \}} {\rho_n(x, y)} \,\d \mm(y) \bigg) \,\d \mm(x)\\
&\geq& \int_{X} |u(x)|^p \bigg( \int_{ B^c_{4\d(x, x_0)\vee R}(x)} {\rho_n(x, y)} \,\d \mm(y) \bigg) \,\d \mm(x)
\end{eqnarray*}
and
\[
 II_b(R,n) \leq \int_{X} |u(x)|^p \bigg( \int_{ B^c_{\d(x, x_0)\vee R}(x)} {\rho_n(x, y)} \,\d \mm(y) \bigg) \,\d \mm(x).
\]
By Assumption \ref{assumption1}-A) and Fatou's lemma again
\[\begin{split}
    L\|u\|_{L^p}^p&=\int_X|u(x)|^p \bigg(\lmt{R}{+\infty}\lmti{n}{\infty}\int_{ B^c_{4\d(x, x_0)\vee R}(x)} {\rho_n(x, y)} \,\d \mm(y) \bigg)\,\d\mm(x)\\ 
   &\leq \lmt{R}{+\infty}\lmti{n}{\infty}II_b(R,n)\leq \lmt{R}{+\infty}\lmts{n}{\infty} II_b(R,n)\\
    &\leq \lmt{R}{+\infty}\lmts{n}{\infty}\int_{X} |u(x)|^p \bigg( \int_{ B^c_{\d(x, x_0)\vee R}(x)} {\rho_n(x, y)} \,\d \mm(y) \bigg) \,\d \mm(x)\\
    &\leq \int_{X} |u(x)|^p \lmt{R}{+\infty}\lmts{n}{\infty} \bigg( \int_{ B^c_{\d(x, x_0)\vee R}(x)} {\rho_n(x, y)} \,\d \mm(y) \bigg) \,\d \mm(x)= L\|u\|_{L^p}^p.
\end{split}\]
Hence
\begin{equation}\label{eq2:MS}
\lmt{R}{+\infty} \lmti{n}{\infty} II_b(R, n) = L\|u\|_{L^p}^p=\lmt{R}{+\infty} \lmts{n}{\infty} II_b(R, n).
\end{equation}
Now note that, by the symmetry of $x$ and $y$, 
\[\begin{split}
    II(R,n)&:=\int_{{B} } {|u(x)-u(y)|^p}{\rho_n(x, y)} \,\d \mm(x)\,\d \mm(y)\\ &=2\int_{D:=\big\{(x,y):\,\d(x,y)\geq R,\, \d(y,x_0)>2\d(x,x_0)\big\}} {|u(x)-u(y)|^p}{\rho_n(x, y)} \,\d \mm(x)\,\d \mm(y).
\end{split}\]
Hence \[\begin{split}
    |II(R,n)-2II_b(R,n)|
    &\leq 2\int_{D} {\Big||u(x)-u(y)|^p-|u(x)|^p\Big|}\,{\rho_n(x, y)} \,\d \mm(x)\,\d \mm(y).
\end{split}\]
If $p=1$, by the triangle inequality, the last integral is bounded by $II_a(R,n)$. Hence, by (\ref{eq1.11:MS}),
\begin{equation}\label{diff}
    \lmt{R}{+\infty} \lmts{n}{\infty} |II(R,n)-2II_b(R,n)|=0.
\end{equation}
In case $p>1$, we use Lagrange theorem and find, for every $x,y\in X$, a number $c\in\R$ lying between $u(x)$ and $u(x)-u(y)$ such that
\begin{equation*}
\Big||u(x)-u(y)|^p-|u(x)|^p\Big|=p\,|c|^{p-1}|u(y)|.
\end{equation*}
In particular $|c|\leq \max\{|u(x)|,|u(x)-u(y)|\}$. Therefore, using also Young's inequality for products, 
\[\begin{split}
    |II(R,n)&-2 II_b(R,n)|\leq 2p \int_D \Big(|u(x)|^{p-1}+|u(x)-u(y)|^{p-1}\Big)|u(y)|\,{\rho_n(x, y)} \,\d \mm(x)\,\d \mm(y)\\
    &\leq 2\int_D\Big(\eps^\frac{1}{p-1}(p-1)(|u(x)|^p+|u(x)-u(y)|^p)+\frac{2}{\eps}|u(y)|^p\Big)\,{\rho_n(x, y)} \,\d \mm(x)\,\d \mm(y)
\end{split}\]
for every $\eps>0$. 
Using Minkowski inequality $|u(x)-u(y)|^p\leq 2^{p-1}(|u(x)|^p+|u(y)|^p)$ and passing to the limits, exploiting (\ref{eq1.11:MS}) and 
(\ref{eq2:MS}) we finally get
\[\lmt{R}{+\infty} \lmts{n}{\infty} |II(R,n)-2II_b(R,n)|\leq 2(p-1)\eps^\frac{1}{p-1}(1+2^{p-1})L\|u\|_{L^p}^p.\]
By arbitrariness of $\eps$, we deduce (\ref{diff}) for every $p\geq 1$. Now we estimate
\[\begin{split}
    \lmts{n}{\infty} II(R,n)&\leq \lmts{n}{\infty}\Big(II(R,n)-2II_b(R,n)\Big)+2\lmts{n}{\infty}(II_b(R,n))\\
    &\leq \lmts{n}{\infty}|II(R,n)-2II_b(R,n)|+2\lmts{n}{\infty}(II_b(R,n)).
\end{split}\]
Passing to the limit as $R\to +\infty$, we get by (\ref{diff}) and (\ref{eq2:MS})
\begin{equation}\label{1est}\lmt{R}{+\infty}\lmts{n}{\infty} II(R,n)\leq 2L\|u\|_{L^p}^p.\end{equation}
Similarly,
\[\begin{split}
-\lmti{n}{\infty} II(R,n)&=\lmts{n}{\infty}(-II(R,n))
\\&\leq \lmts{n}{\infty}\Big(-II(R,n)+2II_b(R,n)\Big)+2\lmts{n}{\infty}(-II_b(R,n))\\&\leq \lmts{n}{\infty}|II(R,n)-2II_b(R,n)|-2\lmti{n}{\infty}(II_b(R,n)).
\end{split}\]
Again, as $R\to +\infty$, we get by (\ref{diff}) and (\ref{eq2:MS})
\begin{equation}\label{2est}\lmt{R}{+\infty}\lmti{n}{\infty}II(R,n)\geq 2L\|u\|_{L^p}^p.\end{equation}
By (\ref{1est}) and (\ref{2est}) we conclude
\begin{equation}\label{eq4:MS}
    \lmt{R}{+\infty}\lmti{n}{\infty}II(R,n)= 2L\|u\|_{L^p}^p= \lmt{R}{+\infty}\lmts{n}{\infty}II(R,n).
\end{equation}
Concerning $III(R,n)$, observe that 
\begin{equation}	\label{10}
\begin{aligned}
C&\subset \left\{(x,y)\in X\times X~|~\d(x,y)\geq R, ~\d(x,x_0)\geq \frac{R}{3}\right\},\\C&\subset \left\{(x,y)\in X\times X~|~\d(x,y)\geq R, ~\d(y,x_0)\geq \frac{R}{3}\right\} .\end{aligned}
\end{equation}
Thus, by Minkowski inequality, Fubini's theorem and the symmetry of $\rho_n$, we deduce
\begin{eqnarray*}
III&= & \int_{{C} } {|u(x)-u(y)|^p}{\rho_n(x, y)} \,\d \mm(x)\,\d \mm(y)\\
&\leq & 2^{p-1}\bigg(\int_C|u(x)|^p\rho_n(x,y)\,\d \mm(x)\,\d\mm(y)+\int_C|u(y)|^p\rho_n(x,y)\,\d \mm(x)\,\d\mm(y)\bigg)\\
&\leq & 2^{p-1}\int_{\d(x, x_0)\geq  \frac R3} |u(x)|^p \bigg ( \int_{\d(x, y)\geq  R}  \rho_{n}(x, y)  \,\d \mm(y)\bigg ) \,\d \mm(x) \\
&+& 2^{p-1}\int_{\d(y, x_0)\geq  \frac R3} |u(y)|^p \bigg ( \int_{\d(x, y)\geq  R}  \rho_{n}(x, y)  \,\d \mm(x)\bigg ) \,\d \mm(y)\\
&=& 2^{p}\int_{\d(x, x_0)\geq  \frac R3} |u(x)|^p \bigg ( \int_{\d(x, y)\geq  R}  \rho_{n}(x, y)  \,\d \mm(y)\bigg) \,\d \mm(x).
\end{eqnarray*}
Applying Fatou's lemma as before,
$$ \lmts{n}{\infty} III(R, n) \leq 2^{p}\int_{\d(x, x_0)\geq  \frac R3} |u(x)|^p  \lmts{n}{\infty}  \bigg ( \int_{\d(x, y)\geq  R}  \rho_{n}(x, y)  \,\d \mm(y)\bigg ) \,\d \mm(x)$$
and 
\begin{equation}\label{eq5:MS}\lmt{R}{+\infty} \lmts{n}{\infty} III(R, n)=0.\end{equation}
Combining \eqref{eq5:MS} with \eqref{eq1.1:MS} and  \eqref{eq4:MS} we get the conclusion.
\end{proof}
\begin{remark}
     In the case of a metric measure space with bounded diameter, Assumption \ref{assumption1} C) is always satisfied with $C=0$ and the same holds for Assumption \ref{assumption1} A) with $L=0$. 
     In particular, in the proof of Theorem \ref{th:MS2} we can fix $R>\diam(X)$, so that $B=C=\emptyset$ and 
\[\begin{split}\E_n(u)&:= \iint_{\{(x, y):x, y\in X, x\neq y\}} {|u(x)-u(y)|^p} \rho_{n}(x, y)\,\d\mm(x) \d\mm(y)\\
&= \iint_{\{(x, y):0<\d(x,y)<R\}} {|u(x)-u(y)|^p} \rho_{n}(x, y)\,\d\mm(x) \d\mm(y),\end{split}\]
which goes to $0=2L\|u\|_{L^p}^p$ as $n\to\infty$ by Assumption \ref{assumption1} B).
\end{remark}

 \subsection{Applications and Examples}\label{sect:appl}
In this section we introduce some examples satisfying Assumption \ref{assumption1} and Theorem \ref{th:MS2}. Doing so, we not only extend Maz'ya and  Shaposhnikova's formula on Euclidean spaces (as well as Ludwig's theorem on finite dimensional Banach spaces) to a large family of mollifiers on curved spaces, but  also  find some new sharp functional inequalities.
 
 \subsubsection*{Volume growth conditions}
            
       
  
From now on, $V:[0,+\infty)\to[0,+\infty)$ will be a $C^1$ strictly increasing function such that $\lmt{t}{\infty} V(t)=+\infty$. We will also use the notation $S(t):=\frac{\d}{\dt}V(t)$.\\
Some prototype examples (see e.g \cite{sturm}) arising from Riemannian geometry are for instance given by the family of functions $V^{K, N}(t)$ on $[0, +\infty)$  defined by
 \begin{equation}\label{impV}
V^{K, N}(t):=\left\{
         \begin{array}{lll}
            \displaystyle  t^N~~&& \text{if}~K=0,\quad N=1,2,\dots\\
            
            \displaystyle \int_0^t\sinh^{N-1}\left(r\sqrt{\frac{-K}{N-1}}\right)\,\d r~~&& \text{if}~K<0,\quad N=2,3,\dots\\
            \end{array}  
        \right.
\end{equation}
       
   \begin{definition}[Generalized Bishop--Gromov inequality]\label{assumption2}
   We say that a  metric measure space $\ms$ satisfies the  \emph{generalized Bishop--Gromov inequality} associated to the function $V:[0,+\infty)\to[0,+\infty)$ if, for every $x_0\in X$,
\begin{equation*}
 \frac {\mm\big (B_R(x_0)\big)} {V(R)} \leq \frac {\mm\big (B_r(x_0)\big)} {V(r)}~~~~\text{for all}~~0< r\leq R<+\infty. 
\end{equation*}
   \end{definition}
  \medskip
  
\begin{definition} \label{def:avr}
We say that a  metric measure space $\ms$ admits  the \emph{generalized asymptotic volume ratio} associated to the function $V:[0,+\infty)\to[0,+\infty)$ if there exists ${\rm AVR}_{\ms}^{V}\in [0,+\infty]$ s.t., for every $x_0\in X$,
\begin{equation*}
{\rm AVR}_{\ms}^{V}=\mathop{\lim}_{r \uparrow +\infty} \frac {\mm\big (B_r(x_0)\big)} {V(r)}.
\end{equation*}
\end{definition}
\medskip
\begin{definition} \label{def:den}
Let $\ms$ be  a  metric measure space.  Its  density function $\theta^V$ associated to $V$ is defined as
\begin{equation*}
\theta^V(x):=\mathop{\lim}_{r \downarrow 0} \frac {\mm\big (B_r(x)\big)} {V(r)}~\qquad\forall x\in X.
\end{equation*}
\end{definition}

\begin{remark}
For spaces satisfying the generalized Bishop--Gromov inequality, we can see that ${\rm AVR}_{\ms}^V$ and $\theta^V$ are both well-defined. In particular, generalized Bishop--Gromov inequality implies that the map $r \mapsto  \frac {\mm\big (B_r(x)\big)} {V(r)}$ is non-increasing, so that ${\rm AVR}_{\ms}^{V}>0$ implies that $\theta(x)>0$ for any $x\in X$,  while $\theta\in L^\infty$ implies that $\frac {\mm\big (B_r(x)\big)} {V(r)}$ is uniformly bounded in $x\in X$ and $r>0$.
\end{remark}
\begin{remark}
It is very easy to construct examples of spaces with bounded density satisfying the generalized Bishop--Gromov inequality associated to a function $V$ which is different from the standard models $V^{K,N}$.\\ Consider for instance $X=\R$, equipped with the Lebesgue measure $\mathcal L^1$. \linebreak Let $V:[0,+\infty)\to[0,+\infty)$ be invertible and of class $C^1$ (in particular $V(0)=0$,\linebreak $\lim_{x\to\infty}V(x)=+\infty$ and $V$ is strictly increasing). Suppose in addition that $V^{-1}:[0,+\infty)\to[0,+\infty)$ is sub-additive, which means
\[V^{-1}(a+b)\leq V^{-1}(a)+V^{-1}(b).\]
Standard examples of such maps are $V(x)=x^\alpha$ (with $\alpha\geq 1)$ or $V(x)=e^x-1$.
Under this conditions it is easy to see that the map $\d:X\times X\to[0,+\infty)$
\[\d(x,y):=V^{-1}(|x-y|)\]
is a distance on $X$.
Moreover, for each $x\in X$, 
\[\begin{split}
    &\frac{\mathcal L^1(B_r(x))}{V(r)}=\frac{\mathcal L^1(\{y\in\R: V^{-1}(|x-y|)<r\})}{V(r)}\\
    &=\frac{\mathcal L^1(\{y\in\R:|x-y|<V(r)\})}{V(r)}=\frac{2V(r)}{V(r)}=2.
\end{split}\]
In particular the space $(\R, \d, \mathcal L^1)$ has density $\theta^V\equiv 2$ and satisfies (trivially) the generalized Bishop--Gromov inequality.
\end{remark}

\begin{example}\label{example}
The range of spaces satisfying the generalized Bishop--Gromov inequality and admitting the generalized
asymptotic volume ratio  is actually pretty wide. We list below some relevant examples:

\begin{itemize} 
 
\item [1)]  {\bf Euclidean spaces}: It is known that ${\rm AVR}^{V^{0,N}}_{(\R^N, |\cdot|, \mathcal L^N)}=\omega_N=\frac {|S^{N-1}|} N$ where $\omega_N=\frac{\pi^{\frac N 2}}{\Gamma(\frac N2+1)}$ denotes the volume of an $N$-dimensional unit ball and $|S^{N-1}|$ denotes its surface measure. This is the case studied in Maz'ya-Shaposhnikova's original paper \cite[Theorem 3]{MS}.

\item [2)]  {\bf Finite dimensional Banach spaces}:  Let $(\R^N, \| \cdot \|, \mathcal L^N)$ be an $N$-dimensional Banach space. It can be seen  ${\rm AVR}^{V^{0,N}}_{(\R^N, \|\cdot\|, \mathcal L^N)}$ is the volume of its unit ball. This is the case studied by Ludwig  \cite[Theorem 2]{Ludwig14}.

\item [3)] {\bf  Riemannian manifolds}: 
Let $(M^N, g)$ be a complete Riemannian manifold of dimension $N$ with ${\rm Ric}\geq K$. 
Let $\d$ and $\mm$ be the Riemannian distance  and the volume measure determined by $g$ respectively.
By the  classical Bishop\textendash Gromov volume comparison theorem, the generalized Bishop--Gromov inequality and the generalized asymptotic volume ratio associated to $V^{K,N}$ exist. In particular, when $K=0$, ${\rm AVR}^{V^{0, N}}_{(M^N, \d, \mm)}<+\infty$ if and only if $(M^N, \d, \mm)$ has Euclidean volume growth.

\item[4)] {\bf Carnot groups}: Let $\mathbb{G}=(\mathbb{R}^d,\cdot)$ be a Carnot group of step $s$ endowed with the Carnot\textendash Carath\'eodory distance $\d_{cc}$ and the Lebesgue measure $\mathcal{L}^d$. It is well known that
$\mathcal{L}^d(B_r(x))=r^{Q} \mathcal{L}^d(B_1(0))$ where $Q\in \N$ is the so called homogeneous dimension of $\mathbb{G}$. It is then clear that then ${\rm AVR}^{V^{0,Q}}_\mathbb{G}=\mathcal{L}^d(B_1(0))>0$.

\item[5)] {\bf $\mcpkn$ spaces}:
Let $\ms$ be a metric measure space satisfying the so-called Measure Contraction Property {\rm MCP}$(K, N)$,  a property  introduced independently by Ohta \cite{Ohta-MCP} and Sturm \cite{S-O2},  as a generalization of $\cdkn$ metric measure spaces.
By generalized Bishop\textendash Gromov volume growth inequality (cf. \cite[Theorem 2.3]{S-O2}) on {\rm MCP}$(K, N)$ spaces,    the (generalized) asymptotic volume ratio is well-defined.

\end{itemize}
\end{example}
\medskip
 \subsubsection*{Mollifiers of radial type}
Let $\ms$ be a metric measure space and let $(\rho_n)_n$ be a family of functions defined on $\{(x,y)\in X\times X:x\neq y\}$. We are now going to introduce some conditions which are easier and more manageable than the ones in Assumption \ref{assumption1}.
 \begin{condition}[Approximation of the identity]\label{assumption3}
 We say that $(\rho_n)_{n \in \N}$ are mollifiers satisfying \emph{approximation of the identity of radial type} associated to an increasing map $V:[0,+\infty)\to[0,+\infty)$, if there exists a  sequence of  strictly decreasing functions $\tilde\rho_n\in C^1(0, +\infty)$ such that
 \begin{itemize}
\item [A)] ({\bf Radial distribution}) 
\begin{equation}\label{limit}
\lim_{n\to\infty} \tilde\rho_n(r)=0~~~\forall~ r\in (0,+\infty),
\end{equation} 

\begin{equation}\label{limit2}
\lim_{r\to +\infty} \tilde\rho_n(r) {V}(r) =0~~~\forall ~ n\in \N,
\end{equation}
and 
\[
\rho_n (x, y)=\tilde\rho_n \big( \d(x, y)\big)~~~\forall \ x,y\in X,\ x\neq y.
\]

\item [B)] ({\bf Monotonicity}) For any $n,m\in\mathbb{N}$ with $n >m$, 
\[
(0, +\infty) \ni  r \to \frac{\tilde \rho_n(r)}{\tilde \rho_m(r)} ~~\text{is non-decreasing}.
\]

\item [C)]  ({\bf Approximation of the identity}) 
It holds 
 \begin{equation}\label{eq1:assumption3}
\lmt{R}{+\infty}\lmt{n}{\infty}\int_{R}^{+\infty}  S(r)\tilde \rho_n(r) \,\d r=1,
 \end{equation}
 where $S(t):=\frac{\d}{\d t}V(t)$.
\end{itemize}
\end{condition}

\medskip
\begin{lemma}\label{lemma1}
Let $\ms$ be a metric measure space admitting generalized asymptotic volume ratio ${\rm AVR}_{\ms}^{V}$ and suppose also that there exists a constant $k>0$ such that $\mm(B_r(x))\leq k\, V(r)$ for every $x\in X$ and every $r>0$ sufficiently large. Let $(\rho_n)_{n\in\N}$ be mollifiers satisfying approximation of the identity of radial type associated to $V$. 
Then $(\rho_n)_{n\in\N}$ satisfy Assumption \ref{assumption1} with $L={\rm AVR}_{\ms}^{V}$.
\end{lemma}
\begin{proof} 
{\sl Assumption \ref{assumption1}-A)}:

First of all, notice that ${\rm AVR}_{\ms}^V$ is finite by our assumption on $k$.
For simplicity, we assume that ${\rm AVR}_{\ms}^{V}>0$, the case for ${\rm AVR}_{\ms}^{V}=0$ can be proved in the same way.  
Fix $x\in X$.  For any $\epsilon>0$, there is $R_0>0$ such that 
\begin{equation}\label{avr}(1-\epsilon){\rm AVR}_{\ms}^{V}  {V}(r) \leq \mm\big(B_r(x) \big) \leq (1+\epsilon) {\rm AVR}_{\ms}^{V} {V}(r),~~~\forall r\geq R_0.\end{equation}
For simplicity, we  write
\[
\mm\big(B_r(x) \big) =\big (1+O(\epsilon)\big) {\rm AVR}_{\ms}^{V}  {V}(r).
\]
Let $R\in (R_0, +\infty)$. For any $n\in \N$ we define  $\bar \rho_{n,R}: X\times X\to (0, \tilde \rho_n(R)]$ as
\[
\bar \rho_{n,R}(x, y):=\left\{\begin{array}{ll}
 \rho_n(x, y)~~~\d(x,y)\geq R,\\
\\
\tilde \rho_n(R)~~~~~~0\leq \d(x,y)< R.
\end{array}
\right.
\] 
By  Cavalieri's formula (cf. \cite[Chapter 6]{AT-T}) we can write
\begin{eqnarray*}
&&\int_{B^c_R(x) }  \rho_n(x, y) \,\d \mm(y)+\tilde \rho_n(R) \mm\big(B_R(x) \big) \\
&=& \int_{X }  \bar \rho_{n,R}(x, y) \,\d \mm(y) \\
(\text{Cavalieri's formula})~&=&  \int_{0}^{+\infty} \mm \big(\{\bar \rho_{n,R}(x, y)>r\} \big) \,\d r\\
&=& \int_0^{\tilde \rho_n(R)} \mm\big(B_{\tilde \rho_n^{-1}(r)}(x) \big) \, \d r\\
&=& \int_0^{\tilde \rho_n(R)}  \big (1+O(\epsilon)\big) {\rm AVR}_{\ms}^{V}{V}\big({\tilde \rho_n^{-1}(r)} \big) \, \d r\\
(\text{let} ~t=\tilde \rho_n^{-1}(r)) ~~&=&\big (1+O(\epsilon)\big)  {\rm AVR}_{\ms}^{V} \int_{+\infty}^{R} {V}(t)  \tilde \rho'_n(t)\, \d t\\
&=&  \big (1+O(\epsilon)\big)  {\rm AVR}_{\ms}^{V} \left(\tilde \rho_n(R) {V}(R) -\int_{+\infty}^{R}  {S}(t)  \tilde \rho_n(t)\, \d t \right)
\end{eqnarray*}
where we use integration by parts formula and \eqref{limit2} in the last equality. 
So, using \eqref{avr},
\begin{eqnarray*}
&&\int_{B^c_R(x) }  \rho_n(x, y) \,\d \mm(y)\\&=&2O(\epsilon) {\rm AVR}_{\ms}^{V} \tilde \rho_n(R) {V}(R) +\big (1+O(\epsilon)\big) {\rm AVR}_{\ms}^{V}\int_{R}^{+\infty}  S(r)\tilde \rho_n(r) \,\d r.
\end{eqnarray*}
Letting $n\to \infty,\, R\to +\infty$ and $\epsilon \to 0$, by \eqref{limit} and \eqref{eq1:assumption3},  we get
\[\lmt{R}{+\infty} \lmts{n}{\infty}\int_{B^c_R(x) }  \rho_n(x, y) \,\d \mm(y)=  {\rm AVR}_{\ms}^V=\lmt{R}{+\infty} \lmti{n}{\infty}\int_{B^c_R(x) }  \rho_n(x, y) \,\d \mm(y)\]
which is the thesis.

\medskip

{\sl Assumption \ref{assumption1}-B)}:

Let $u\in L^p(X)$ be such that $\E_{n_0}(u)<+\infty$ for some $n_0\in\N$. For any $R>0$ and any $n>n_0$, by  assumption B) of Condition \ref{assumption3}, it holds
\begin{eqnarray*}
&&  \int_{X} \left( \int_{B_R(y)\setminus\{y\}} {|u(x)-u(y)|^p}{\rho_{n_0}(x, y)} \frac{\rho_{n}(x, y)}{\rho_{n_0}(x, y)}\,\d \mm( x)\right )\,\d \mm(y)\\
&\leq &   \int_{X} \left( \int_{B_R(y)\setminus\{y\}} {|u(x)-u(y)|^p}{\rho_{n_0}(x, y)} \frac{\tilde \rho_{n}(R)}{\tilde\rho_{n_0}(R)}\,\d \mm( x) \right )\,\d \mm(y)\\
&\leq&  \E_{n_0}(u)  \frac{\tilde \rho_{n}(R)}{\tilde\rho_{n_0}(R)}.
\end{eqnarray*}
Then \eqref{eq2:assumption1} follows from \eqref{limit}.
\medskip 

{\sl Assumption \ref{assumption1}-C)}:

Proceeding as in the first part of the proof, we get that, for $R$ sufficiently large,
\[\begin{split}
    \int_{B^c_R(x)}\rho_n(x,y)\,\d\mm (y)&= \int_0^{\tilde\rho_n(R)}\mm(B_{\tilde\rho_n^{-1}(r)}(x))\,\d r-\tilde\rho_n(R)\mm (B_R(x))\\
    &\leq \int_0^{\tilde\rho_n(R)}\mm(B_{\tilde\rho_n^{-1}(r)}(x))\,\d r\\
    &\leq k\int_0^{\tilde\rho_n(R)}V(\tilde\rho_n^{-1}(r))\,\d r\\
    &=-k\int_R^{+\infty}V(s)\,\tilde\rho_n'(s)\,\d s\\
    &\stackrel{\eqref{limit2}}=k\left[V(R)\,\tilde\rho_n(R)+\int_R^{+\infty}S(s)\tilde\rho_n(s)\,\d s\right]
\end{split}\]
The last quantity does not depend on $x$ and, by assumptions \eqref{limit} and \eqref{eq1:assumption3}, it is bounded in $n$, completing the proof.
\end{proof}
\begin{remark}
    Note that the conditions on the space required in Lemma \ref{lemma1} (and in the following Theorem \ref{th:MS}) are fulfilled for instance if $\ms$ satisfies a generalized Bishop--Gromov inequality and has bounded density function. In the final examples we will often refer to these stronger conditions.
\end{remark}
\begin{remark}
    Notice also that a space with bounded diameter admits always generalized asymptotic volume ratio and, if $\mm (X)<+\infty$,  then ${\rm AVR}_{\ms}^V=0$ necessarily. Moreover, the proof of Lemma \ref{lemma1} reduces to the task of proving Assumption \ref{assumption1}-B), since $L={\rm AVR}^V_{\ms}=0$ and A) and C) are automatically satisfied, as already observed. Notice that in this case also the second condition required in Lemma \ref{lemma1} is always fulfilled.
\end{remark}
\bigskip
 
We construct now an interesting class of mollifiers satisfying  approximation of the identity of radial type in Condition \ref{assumption3}.

\begin{lemma}\label{examples} 
Let $(a_n)_{n\in \N}$ be a strictly decreasing sequence converging to 0 and let $f:(0,+\infty)\to\R$ be a positive increasing $C^2$ function which satisfies the following properties: 
\begin{equation}\label{1ip}\lim_{s\to\infty}f(s)=+\infty;\end{equation}
\begin{equation}\label{2ip}
    \lim_{s\to\infty}\frac{sf'(s)}{[f(s)]^r}=0\quad\text{for }r>1;
\end{equation}
\begin{equation}\label{3ip}
    s\mapsto \frac{f'(s)}{[f(s)]^r} \quad \text{is strictly decreasing for }r>1.
\end{equation}
 Then the mollifiers $\tilde \rho_n:(0,+\infty)\to\mathbb R$ defined by 
\begin{equation}\label{mollifiers}\tilde \rho_n(t):=\frac{a_n\,f'(V(t))}{[f(V(t))]^{a_n+1}}\end{equation}
satisfy the requests in Condition \ref{assumption3}.
\end{lemma}
\begin{proof}
First of all, $\tilde\rho_n$ are strictly decreasing: this follows by \eqref{3ip} since $a_n>0$ and $V$ is a strictly increasing function.\\
Moreover \eqref{limit} follows by the fact that $a_n\to 0$, while we get \eqref{limit2} since
\[\tilde \rho_n(t)V(t)=\frac{a_n\,f'(V(t) )V(t)}{[f(V(t))]^{a_n+1}}\]
combined with \eqref{2ip} and the fact that $V(t)\to+\infty$ as $t\to\infty$.\\
Regarding condition B), for $n>m$ we compute
\[\frac{\tilde\rho_n(t)}{\tilde\rho_m(t)}=\frac{a_n}{a_m}[f(V(t))]^{a_m-a_n},\]
which is non-decreasing since $a_n<a_m$ and both $f$ and $V$ are increasing functions. Finally, we conclude by observing
\[\int_\delta^{+\infty} S(t)\tilde\rho_n(t)\,dt= \Big[-[f(V(t))]^{-a_n}\Big]_\delta^{+\infty}=[f(V(\delta))]^{-a_n}\stackrel{n\to+\infty}\longrightarrow 1.\qedhere\]
\end{proof}
\begin{remark}
Conditions \eqref{1ip}-\eqref{3ip} are satisfied by plenty of examples: for instance $f(s)=s^\alpha$ for $\alpha>0$, $f(s)=e^s$ or $f(s)=\ln(s)$ (on the half-line $(1,+\infty)$).\\
In particular, considering  $f(s)=s^{\alpha}$ we obtain the family of mollifiers
\begin{equation}\label{es1}\tilde\rho_n(t)=\frac{a_n\,\alpha}{[V(t)]^{\alpha a_n +1}}\qquad \alpha>0.\end{equation}
Notice that, if we choose $V(t)=V^{0,N}(t)=t^N$ and $\alpha=1/N$, we get $\tilde \rho_n(t)=\frac{a_n}{Nt^{a_n+N}}$, which is the standard family of mollifiers which appears in \eqref{MS:intro}.\\
In addition, considering $f(s)=e^s$, we get the mollifiers
\begin{equation}\label{es2}\tilde\rho_n(t)=\frac{a_n}{e^{a_n V(t)}},\end{equation}
while taking $f(s)=\ln s$ we derive

\begin{equation}\label{es3}\tilde\rho_n(t)=
                      \frac{a_n}{\big(\ln (V(t))\big)^{a_n+1} V(t)}\end{equation}
                      on the half-line $(V^{-1}(1),+\infty)$.

\end{remark}

\subsubsection*{Asymptotic formulas}
Combining Theorem \ref{th:MS2} and Lemma \ref{lemma1},  we   get the following asymptotic formulas.

\begin{theorem}\label{th:MS}Let $\ms$ be a metric measure space admitting generalized asymptotic volume ratio ${\rm AVR}_{\ms}^V$ and suppose also there exists a constant $k>0$ such that 
\begin{equation}\label{volumebound}
    \mm(B_r(x))\leq k\, V(r)
\end{equation}
for every $x\in X$ and every $r>0$ sufficiently large. Let $(\rho_n)_{n\in\N}$ be mollifiers of radial type satisfying Condition \ref{assumption3}. 
Then for any  $u \in L^p$ such that
$\E_{n_0}(u)<+\infty$ for a certain $n_0\in \N$, it holds 
\begin{equation*}
\mathop{\lim}_{n\to \infty} \E_n(u)=2{\rm AVR}_{\ms}^V\|u\|^p_{L^p}.
\end{equation*}
\end{theorem}
\medskip
\begin{remark}
 The assumptions in Theorem \ref{th:MS} are very general and they encompass several important examples. Indeed, as already remarked in Example \ref{example}, finite dimensional Banach spaces, Carnot groups, Riemannian manifolds with Ricci curvature bounded from below and MCP spaces all satisfy the general Bishop--Gromov inequality and admit the generalized asymptotic volume ratio. On the other hand a direct computation show that also \eqref{volumebound} is satisfied in Carnot groups and {finite dimensional Banach spaces}. Moreover, Theorem \cite[Theorem III.4.4]{Riem} and \cite{DPG-Non} show that \eqref{volumebound} holds in Riemannian manifolds with Ricci curvature bounded from below and non-collapsed RCD spaces respectively taking $V$ as in \eqref{impV}.
\end{remark}
\begin{remark}
Observe also that, if the space has finite measure, then assumptions in Theorem \ref{th:MS} are always satisfied and 
\[\lim_{n\to\infty}\E_n(u)=0\]
for every $u\in L^p(X)$ with $\E_{n_0}(u)<+\infty$ for a certain $n_0\in\N$.\\
For instance this holds if we consider the sphere $\mathbb S^n$ with the standard mollifiers used in \eqref{MS:intro}: this case is also discussed in \cite{kreuml}.
\end{remark}
In particular we can apply Theorem \ref{th:MS} to the families of mollifiers described in Lemma \ref{examples}. For instance, considering \eqref{es1} with $a_n=sp$, we get the following asymptotic formula:
\begin{example}
    Let $\ms$ be a metric measure space satisfying the generalized Bishop--Gromov inequality associated to $V$. Assume that its associated density function is bounded.  Then for any  $p\geq 1$, $\alpha>0$, if there exists $s_0\in (0, 1)$ and $u\in L^p\ms$ such that
\[
 \iint_{\{(x,y)\in X\times X: x\neq y\}} \frac{|u(x)-u(y)|^p}{[V(\d(x, y))]^{\alpha s_0 p+1}} \,\d \mm( x)\,\d \mm(y)<\infty,
\]
we have
\[
\mathop{\lim}_{s\downarrow 0} s  \iint_{\{(x,y)\in X\times X:x\neq y\}} \frac{|u(x)-u(y)|^p}{[V(\d(x, y))]^{\alpha sp+1}} \,\d \mm( x)\,\d \mm(y) =\frac {2}{\alpha p} {\rm AVR}^{V}_{\ms} \|u\|^p_{L^p}.
\]
In particular, taking $X=\R^N, V(t)= V^{0,N}(t)=t^N$ and $\alpha=1/N$, we recover exactly Maz’ya and Shaposhnikova's asymptotic formula \eqref{MS:intro} with $L=\frac{2N\omega_N}{p}$.
\end{example}
\medskip
Similarly, if we consider \eqref{es3} with $a_n=sp$, we get the following similar result. Notice that in this case we need to restrict to the set $\{(x,y):V\big(\d(x,y)\big)>1\}$, since the function $f$ is defined on $(1,+\infty)$. Nevertheless, the same proof of Theorem \ref{th:MS2} still works with the obvious changes.
\begin{example}\label{coro:MS}
Let $\ms$ be a metric measure space with bounded density, satisfying the generalized Bishop--Gromov inequality associated to $V$.  For any  $p\geq 1$, if there exists $s_0\in (0, 1)$ and $u\in L^p\ms$ such that
\[
 \iint_{V\big(\d(x, y)\big)>1} \frac{|u(x)-u(y)|^p}{\big[\ln V\big(\d(x, y)\big)\big]^{s_0p+1}  V(\d(x, y))} \,\d \mm( x)\,\d \mm(y)<\infty,
\]
we have
\[
\mathop{\lim}_{s\downarrow 0} s  \iint_{V\big(\d(x, y)\big)>1} \frac{|u(x)-u(y)|^p}{\big[\ln V\big(\d(x, y)\big)\big]^{sp+1} V(\d(x, y))} \,\d \mm( x)\,\d \mm(y)=\frac {2}p {\rm AVR}^{V}_{\ms} \|u\|^p_{L^p}.
\]

\end{example}
Similar results can be obtained by considering mollifiers of different shape, as soon as they satisfy approximation of the identity of radial type, described in Condition \ref{assumption3}.\\\\
We conclude by pointing out another interesting example in which our theory can be applied: cf. \cite{DPG-Non} for the definition of non-collapsed ${\rm RCD}(0,N)$ space.

\bigskip
\begin{example}\label{rigidity1}
Let $(X, \d, \mathcal H^N)$ be a non-collapsed ${\rm RCD}(0, N)$ metric measure space with $N\in \N^*$. Then
\begin{equation}\label{eq1:rigidity1}
\mathop{\lim}_{s\downarrow 0} s\int_{X} \int_{X} \frac{|u(x)-u(y)|^p}{\d^{N+sp}(x, y)} \,\d \mm( x)\,\d \mm(y) \leq \frac {2N\omega_N} p  \|u\|^p_{L^p}
\end{equation}
for any $u\in L^p(X)$ which has finite associated energy for a certain $s=s_0$. 
Here $\omega_N$ denotes the volume of an $N$-dimensional unit ball. \\
If the equality in \eqref{eq1:rigidity1} is attained by a nowhere-zero function $u$, then  $\ms$ is  isometric to a metric cone over an ${\rm RCD}(N-2, N-1)$ space. 
\end{example}

\begin{proof}
Firstly, by Generalized Bishop\textendash Gromov volume growth inequality (cf. \cite[Theorem 2.3]{S-O2}) we know 
\begin{equation}\label{eq2:rigidity1}
 {{\rm AVR}}_{\ms}^{V^{0,N}} =\lmt{r}{\infty}\frac {\mm\big (B_r(x_0)\big)} {r^N}  \leq \lmt{R}{0} \frac {\mm\big (B_R(x_0)\big)} {R^N} = \omega_N,~~~\forall R>0.
\end{equation}
Combining with Theorem \ref{th:MS} we get \eqref{eq1:rigidity1}.\\
If the equality in \eqref{eq1:rigidity1} is attained by a nowhere-zero function $u$, by Theorem \ref{th:MS} and \eqref{eq2:rigidity1} we can see that 
\[
\frac {\mm\big (B_R(x_0)\big)} {R^N} =\omega_N,~~~\forall R>0.
\]
By \cite[Theorem 1.1]{DPG-F} we know $\ms$ is isometric to a metric 
cone over an ${\rm RCD}(N-2, N-1)$ space. 
\end{proof}

\def\cprime{$'$}

\end{document}